\RequirePackage{fixltx2e}
\documentclass[a4paper]{isp}

\usepackage[T1]{fontenc}

\usepackage{lmodern}

\usepackage{amsmath}
\usepackage{url}

\ifnum\pdfoutput>0
\usepackage[
bookmarks=false,
breaklinks=true,
colorlinks=true,
linkcolor=black,
citecolor=black,
urlcolor=black,
pdfpagelayout=SinglePage
]{hyperref}
\usepackage[all]{hypcap}
\else
\usepackage{hyperref}
\fi

\usepackage{latexsym}
\usepackage{amsmath,amsthm}
\usepackage{amsfonts}
\usepackage{amssymb}
\usepackage{graphicx}
\usepackage{color}
\usepackage{epsfig}

\newtheorem{Proposition}{Proposition} 
 
\newtheorem{Lemma}{Lemma}
\newtheorem{Definition}{Definition}
\newtheorem{Corollary}{Corollary}

\newtheorem{Remark}[Proposition]{Remark}
\newtheorem{Theorem}[Proposition]{Theorem}

\begin{document}
\pdfgentounicode=1

\title{Ergodic Theorems for the Transfer Operators of \\ Noisy Dynamical Systems}
	 
\titlerunning{Ergodic Theorems of Noisy Systems}

\author{Eleonora Catsigeras}
\institute{ Instituto de Matem\'{a}tica y Estad\'{\i}stica "Rafael Laguardia", \\ Universidad de la Rep\'{u}blica,  Uruguay \\ \email{eleonora@fing.edu.uy}}




\maketitle

\begin{abstract}
We consider stationary stochastic dynamical systems evolving on a compact metric space, by perturbing a deterministic dynamics with a random noise, added  according to an arbitrary probabilistic distribution. We prove the maximal and pointwise ergodic theorems   for the transfer operators associated to such systems.  The  results  are extensions to noisy  systems of some of the fundamental  ergodic theorems for   deterministic systems. The  proofs are analytic. They follow  the  rigorous deductive method of the classic proofs in pure mathematics.
 
\end{abstract}

\section{Introduction}\label{sectionIntroduction}

The  ergodic theory of deterministic (zero-noise) dynamical systems is  based on a series of classical ergodic theorems. The first  of those theorems is the pointwise ergodic theorem of   Birkhoff-Khinchin \cite{Birkhoff} that ensures the almost sure convergence of the time averages (see also for instance \cite[Chapter II]{Mane}, \cite[Theorem 4.1.2]{Katok-Hasselblatt}, or \cite[Chapter 1, p. 11]{Cornfed-Fomin-Sinai}).
Also the  ergodic theorem of Kingman  \cite{Kingman}  ensures the pointwise convergence of the time averages, in a more general formulation, stating the convergence of any subadditive sequence of functions. For a recent more general statement and proof of the subadditive ergodic theorem see  \cite{Gouezel-Karlsson}.
For differentiable dynamical systems, the multiplicative ergodic theorem of Oseledets (\cite{Oseledets}) is also a fundamental result in the Ergodic Theory of deterministic dynamical systems, since, among other important consequences, it implies the existence of the Lyapunov exponents and gives a powerful tool  for the study of (differentiable) chaos. See also for instance \cite[Sections IV.10 and 11]{Mane}, \cite[Theorem 4.6.2]{yo}, or for a relatively short proof and subsequent generalizations of the multiplicative ergodic theorem, see \cite{Filip}.

The  proofs of the pointwise ergodic theorems are very different now a days  from the original proofs of Birkhoff, and independently of Khinchin, when they first discovered them in the decade of 1930. Mathematicians now deduce  the pointwise ergodic theorem as a particular case of  many other   more general results:  the maximal ergodic theorems \cite[Chapter IV]{Mane}, \cite[Theorems 2.2.5 and 2.2.6] {yo}, the subadditive ergodic theorem of Kingman and its generalization for cocycles \cite{Gouezel-Karlsson}, the operator theoretic ergodic theorems \cite{Eisner}, the entrangled ergodic theorem \cite{Eisner-Kunszenti-Kovacs}, the ergodic theorems for group actions \cite{Tempelman} and for noncommutative products \cite{Karlsson-Ledrappier}, and the ratio ergodic theorems \cite{Rugh-Thomine}, among others. Besides the ergodic theorems were also generalized for infinite measure spaces \cite{Aaronson}, \cite{Roy}. Also the classical   multiplicative ergodic theorem of Oseledets is now generalized in several forms; as for instance in Filip's extensions \cite{Filip},  Austin's multiplicative ergodic theorem \cite{Austin}, \cite{delaRue}, and Gonz\'{a}lez-Tokman-Quass multiplicative ergodic theorem for cocycles  \cite{Gonzalez-Tokman-Quas}.

Some of the results cited above apply only to deterministic (zero-noise) dynamical systems. Precisely, the question motivating this paper is: Are those ergodic theorems applicable or extendable also to stochastic or noisy systems?
In fact, some of them already have adapted statements that are applicable to stochastic processes, Markovian systems, or random transformations systems (RTS). For instance, very early, the pointwise ergodic theorem for Makovian processes was proved by Kakutani \cite{Kakutani} (see also \cite[Theorem 6, p. 388]{Yosida} and \cite[Corollary 2.2, p. 24]{Kifer}). And much later,  the multiplicative ergodic theorem for RTS is stated and proved in \cite[Chapter III, p. 88]{Kifer}, and also in \cite[Chapter 4]{Arnold}.

The purpose of  this paper is precisely to state and prove the maximal and pointwise ergodic theorems for stochastic dynamical systems. They are Theorems \ref{TheoremMaximalErgodic} and \ref{MainTheorem1},  and Corollaries \ref{corollaryMaximalErgTheo} and \ref{corollaryMaximalErgTheo2}.  As a consequence, we also  obtain  a different  proof of Kakutani's ergodic theorem, when applied to a transfer operator (see Theorem \ref{TheoremMainErgodic}).

 \subsection{Setting up} 

Along this paper, $X$ is a compact metric space and  ${\mathcal A}$ is the Borel sigma-algebra  in   $X$. We denote by ${\mathcal M}$   the space of all the probability measures on  $( X, {\mathcal A})$  endowed with the weak$^*$-topology. We  consider the functional space $C^0(X, \mathbb{C})$ of all the continuous functions $\varphi: X \mapsto \mathbb{C}$ with the supremum norm. We denote by $C^0(X, \mathbb{R})$   the subspace of real continuous functions, and by $C^0(X, [0,1])$    its   subset of    functions whose values belong to the interval $[0,1]$.

A stochastic dynamical system on $X$ is a stochastic process $x_0, x_1, \ldots, x_n, \ldots$ with any given initial  probability distribution  $ \mu_0 \in {\mathcal M}$ such that $$ \mu_0(A) = \mbox{prob}\{x_0 \in A\} \ \ \forall \ A \in {\mathcal A},$$ and a family $\{P(x, \cdot\}_{x \in X} \subset {\mathcal M}$ of   \em transition probabilities \em $P(x, \cdot)$ (also called \em probabilities of noise\em) such that
$$ P(x, A) = \mbox{prob}\{x_{n+1} \in A / x_n = x\} \ \ \forall \ n \geq 0, \ \ \forall \ A \in {\mathcal A}, \ \ \forall \ x \in X.$$ The stochastic dynamical system is continuous if the application $x \in X \mapsto P(x, \cdot ) \in {\mathcal M}$ is continuous in the weak$^*$-topology.

When studying the   properties of continuous stochastic dynamical systems, the   following transfer operator $ {\mathcal L}: C^0(X, {\mathbb{C}}) \mapsto C^0(X, {\mathbb{C}})$  and its dual transfer operator $ {\mathcal L}^*: {\mathcal M} \mapsto {\mathcal M}$ are usually considered:
$$({\mathcal L} \varphi)(x) := \int \varphi (y) P(x, dy)  \ \   \forall \ x \in X, \ \ \forall \ \varphi \in C^0(X, {\mathbb{C}});$$
$$ \int \varphi \, d ({\mathcal L}^* \mu) := \int ({\mathcal L} \varphi) \, d \mu \ \ \forall \  \varphi \in C^0(X, {\mathbb{C}}), \ \ \forall \ \mu \in {\mathcal M}.$$
The ergodic properties of the continuous stochastic dynamical system   rely on the   convergence $\mu$-a.e. (when it occurs) of the time averages $ \frac{1}{n}\sum_{j= 0}^{n-1}({\mathcal L}^j \varphi)$, and also, on the properties of the limits of weak$^*$-convergent subsequences of   $ \frac{1}{n}\sum_{j= 0}^{n-1}({{\mathcal L}^*}^j \mu)$.

 We will start by considering any  operator ${\mathcal L}$ from $ C^0(X, \mathbb{C})  $ to itself, that is positive, bounded by 1, and such that ${\mathcal L}(1)= 1$. As said above, along this paper we will study the ergodic properties of the iteration of ${\mathcal L}$, and of its dual operator ${\mathcal L}^*$ on the space ${\mathcal M}$ of probability measures. A priori, ${\mathcal L}$ is not constructed as the transfer operator of a stochastic dynamical system. Nevertheless, in Section \ref{sectionPreviousResults}-Proposition \ref{proposition-P(x,cdot)}, we will show that there exists such a  stochastic  system  whose transfer operator coincides with the given ${\mathcal L}$.

\subsection{Definitions}

\begin{Definition} \label{DefinitionTransferOperatorL}
{\bf (Transfer Operator ${\mathcal L}$  in the space of continuous functions.)} \em

 A \em Transfer Operator ${\mathcal L}$ in the space  of continuous functions \em   is a linear operator $${\mathcal L}: C^0(X, \mathbb{C}) \mapsto  C^0(X, \mathbb{C}) $$ such that:

 \noindent {\bf \ref{DefinitionTransferOperatorL}.1 }  ${\mathcal L}$ is positive; precisely  ${\mathcal L} \varphi$ is real and non negative  if $\varphi$ is real and non negative.

 \noindent {\bf \ref{DefinitionTransferOperatorL}.2 }    $\|{\mathcal L}\| = 1$; precisely $\max_{x \in X}|({\mathcal L} \varphi) (x)| \leq  \max_{x \in X} |\varphi(x)|$ for all $\varphi \in C^0(X, \mathbb{C})$, and  ${\mathcal L}   \cdot 1 = 1$.
\end{Definition}

\begin{Definition}
{\bf (Transfer Operator ${\mathcal L}^*$  in the space of probability measures.)} \em

For any   transfer operator ${\mathcal L}: C^0(X, \mathbb{C}) \mapsto C^0(X, \mathbb{C})$, the \em Dual Transfer Operator ${\mathcal L}^*$ in the space ${\mathcal M}$ of probabilities \em is the application ${\mathcal L}^*: {\mathcal M} \mapsto {\mathcal M}$   defined by

  \begin{equation} \label{eqn000}
  \int \varphi \, d{\mathcal L}^* \mu = \int  {\mathcal L} \varphi  \, d \mu \ \ \ \ \forall \ \varphi \in C^0(X, \mathbb{C}) \ \ \ \ \ \forall \ \mu \in {\mathcal M}.\end{equation}

Due to Riesz Theorem, the dual transfer operator ${\mathcal L}^*$ in the space of probability measures  exists and is unique for any given transfer operator ${\mathcal L}$ in the space of continuous functions.
\end{Definition}

We are particularly interested in those probability measures that are fixed by the dual transfer operator ${\mathcal L}^*$, and more generally, in those probability measures $\mu$ that are ${\mathcal L}^*$-periodic with period $p \geq 1$; i.e. fixed by ${{\mathcal L}^*}^p$ for a minimum natural number $p \geq 1$.

\vspace{.2cm}

\noindent {\bf Extension of the operator ${\mathcal L}$ to the space $L_\infty$.}

In Section \ref{sectionPreviousResults}-Proposition \ref{propositionCalL:L_inftyToL_infty},  we prove that the transfer operator ${\mathcal L}$ can be extended to the space $L_{\infty}$ of bounded measurable functions  in such a way that the following equality holds:
$$
  \int \varphi \, d{\mathcal L}^* \mu = \int  {\mathcal L} \varphi  \, d \mu \ \ \ \ \forall \ \varphi \in L_{\infty} \ \ \ \ \ \forall \ \mu \in {\mathcal M}.$$

\begin{Definition}
 \label{DefinitionInvariantSets} {\bf (Invariant sets almost everywhere)} \em

 Let $\mu \in {\mathcal M}$ and $A \in {\mathcal A}$. We say that $A$ is $\mu$-a.e. ${\mathcal L}$-invariant if
$${\mathcal L}\chi_A = \chi_A \ \ \ \mu-a.e.,$$
where $\chi_A$ denotes the characteristic function of $A$.
Analogously, the measurable set $A$ is $\mu$-a.e. ${\mathcal L}^p$-invariant for a natural number $p \geq 1$, if
$${\mathcal L}^p \chi_A = \chi_A \ \ \ \mu-a.e.$$\end{Definition}

\begin{Definition} \em \label{definitionErgodic} {\bf (Ergodic measures).}

Let $\mu \in {\mathcal M}$. We say that $\mu$ is \em ergodic  for ${\mathcal L}^*$ \em  if it is fixed by ${\mathcal L}^*$ and if $\mu(A) \in \{0,1\}$ for any set $A \subset M$ that is ${\mathcal L}$-invariant $\mu$-a.e. In other words, if
$\chi_A(x) = ({\mathcal L} \chi_A)(x)$  for $\mu$-a.e. $x \in X$, then either $\mu(A) = 0$ or $\mu(A) = 1$.

Analogously, for any natural number $p \geq 1$, we say that   $\mu$ is \em ergodic  for ${{\mathcal L}^*}^p$ \em  if it is fixed by ${{\mathcal L}^*}^p$ and if $\mu(A) \in \{0,1\}$ for any set $A \subset M$ that is ${\mathcal L}^p$-invariant $\mu$-a.e.
\end{Definition}

\subsection{Statement of the Results.}

The   purpose of this paper is to prove the following results:

\begin{Theorem} {\bf (Maximal Ergodic Theorem)} \label{TheoremMaximalErgodic}

Let $\mu \in {\mathcal M}$ such that ${\mathcal L}^* \mu = \mu$. Let $\varphi \in L_{\infty}$ be a real function. Define
\begin{eqnarray}
\nonumber
\varphi_n(x) &:= &\max\{\varphi(x), (\varphi + {\mathcal L} \varphi)(x), \ldots, (\varphi + {\mathcal L} \varphi + \ldots +{\mathcal L}^{n-1} \varphi)(x)\};\\ \nonumber
  E(\varphi) &:= &\{ x \in X \colon \sup_{n \geq 1} \varphi_n(x) > 0\}. \end{eqnarray}
 Then
 \begin{equation}
 \label{eqnTeoErgMaximal}
 \int_{E(\varphi)} \varphi \, d \mu \geq 0. \end{equation}
\end{Theorem}
\begin{Corollary}
\label{corollaryMaximalErgTheo}
Let $\mu \in {\mathcal M}$ such that ${\mathcal L}^* \mu = \mu$. Let $\varphi: X \mapsto \mathbb{R}$ be bounded and measurable. For each natural number $n \geq 1$ and each real number $\alpha$, denote:
\begin{eqnarray}
\nonumber
\varphi_n &:= &\sum_{j= 0}^{n-1} {\mathcal L}^j \varphi.\\  \nonumber
C_{\alpha} & := & \{x \in X \colon  \sup_{n \geq 1} \frac{ \varphi_n (x)}{n} > \alpha \}.\end{eqnarray}
Then, for any measurable set $A \subset C_{\alpha}$ such that $ ({\mathcal L} \chi_A)(x)   = \chi_A (x) $ for $\mu$-a.e. $x \in X$, the following inequality holds:
$$\int_A \varphi \, d \mu \geq \alpha \cdot \mu(A).$$
\end{Corollary}
\begin{Corollary}
\label{corollaryMaximalErgTheo2}
Let $\mu \in {\mathcal M}$ such that ${\mathcal L}^* \mu = \mu$. Let $\varphi: X \mapsto \mathbb{R}$ be bounded and measurable. For each natural number $n \geq 1$ and each real number $\beta$, denote:
\begin{eqnarray}
\nonumber
\varphi_n &:= &\sum_{j= 0}^{n-1} {\mathcal L}^j \varphi.\\ \nonumber
B_{\beta} & := & \{x \in X \colon  \inf_{n \geq 1} \frac{ \varphi_n (x)}{n} < \beta \}.\end{eqnarray}
Then, for any measurable set $A \subset B_{\beta}$ such that $ ({\mathcal L} \chi_A)(x)   = \chi_A (x) $ for $\mu$-a.e. $x \in X$, the following inequality holds:
$$\int_A \varphi \, d \mu \leq \beta \cdot \mu(A).$$

\end{Corollary}

 \begin{Theorem} \label{MainTheorem1} {\bf (Pointwise Ergodic Theorem for Periodic Measures by the Transfer Operator)}

 Let $\mu $ be a probability measure fixed by the transfer operator ${{\mathcal L}^*}^p$ for some natural number $p \geq 1$. Then, for any function $\varphi \in L_{\infty}$ the following limit exists $\mu$-a.e.:
 $$\widetilde \varphi_p (x) := \lim_{n \rightarrow + \infty} \frac{1}{n} \sum_{j=0}^{n-1} {\mathcal L}^{jp} \varphi )(x) \ \ \ \mu\mbox{-a.e. } x \in X. $$
 If besides $\mu$ is ergodic for ${{\mathcal L}^*}^p$, then
 $$\widetilde \varphi_p(x) = \int \varphi \, d \mu \ \  \mu\mbox{-a.e. } x \in X.$$
 \end{Theorem}

  The proofs of Theorems \ref{TheoremMaximalErgodic} and \ref{MainTheorem1}, as well as the proofs of   their corollaries and other ergodic theorems, will be developed along Sections \ref{sectionErgodicTheorems} and \ref{sectionErgodicMeasures}. In Section \ref{sectionPreviousResults}, we  prove some    previous statements.

\section{Previous Results} \label{sectionPreviousResults}

\begin{Proposition} \label{proposition-P(x,cdot)}
Let ${\mathcal L}$ be a transfer operator. Then, there exists a unique family of probability measures $\{P (x, \cdot)\}_{x \in X} \subset {\mathcal M}$ such that:
$$({\mathcal L} \varphi) (x) = \int \varphi(y) \, P (x, dy) \ \ \ \forall \ x \in X \ \ \ \forall \ \varphi \in C^0(X, \mathbb{C}).$$
Besides, the probability measure  $P({x, \cdot}) \in {\mathcal M}$ depends continuously on $x \in X$ in the weak$^*$ topology of ${\mathcal M}$.
\end{Proposition}

\begin{proof}
Fix $x \in X$. The transformation $\Lambda_x :C^0(X, \mathbb{R}) \mapsto   \mathbb{R}$, defined by $ \Lambda_x(\varphi := ({\mathcal L} \varphi) (x)  $ is a  linear operator defined on the space of real continuous functions. It is positive,   bounded by 1, and $\Lambda_x(1)= 1$. So, applying Riesz Theorem there exists a unique probability measure $P(x, \cdot) \in {\mathcal M}$  such that $\Lambda_x \varphi = \int \varphi(y) \, P(x, dy)$ for all $\varphi \in C^0(X, \mathbb{C})$. To end the proof of Proposition \ref{proposition-P(x,cdot)} we must prove that $P(x, \cdot)$  depends continuously on $x$ in the weak$^*$ topology of ${\mathcal M}$. Equivalently, we must prove that if $x_n \rightarrow x \in X$ as ${n \rightarrow + \infty}$ then, $\lim_{n \rightarrow + \infty} \int \varphi(y) P(x_n, dy) = \int \varphi (y) \, P(x, dy)$ for any continuous function $\varphi \in C^0(X, \mathbb{C})$. In fact, by construction of the probability measure $P(x_n, \cdot)$, and recalling that ${\mathcal L} \varphi$ is  by definition a continuous function if $\varphi$ is continuous, we have:
$$\lim_{n \rightarrow + \infty} \int \varphi (y) P(x_n, dy) = \lim _{n \rightarrow + \infty} ({\mathcal L}\varphi)(x_n) = ({\mathcal L} \varphi)(x)  =   \int \varphi(y) P(x, dy), $$
ending the proof.
\end{proof}

\begin{Remark} \label{RemarkInvariantSetsWithP(x,A)}
\em As a consequence of Proposition \ref{proposition-P(x,cdot)}, for any   measure $\mu \in {\mathcal M}$, we have:
\begin{equation}
\label{eqnInvariantSetsLchi=P}
A \in {\mathcal A} \mbox{ is } \mu\mbox{-a.e. }  {\mathcal L}\mbox{-invariant } \ \ \ \Leftrightarrow \ \ \ \chi_A(x) = ({\mathcal L} \varphi) (x) = P(x, A) \ \ \mbox{ for } \mu\mbox{-a.e. } x \in X.\end{equation}
\end{Remark}

\begin{Proposition}
\label{PropositionL*continuous}

The transfer operator ${\mathcal L}^*: {\mathcal M} \mapsto {\mathcal M}$ is continuous in the weak$^*$ topology of ${\mathcal M}$.
\end{Proposition}
\begin{proof}
If $\mu_n \rightarrow \mu$ in the weak$^*$ topology, then $ \int {\mathcal L} \varphi \, d \mu_n \rightarrow \int {\mathcal L} \varphi \, d \mu$ for any continuous function $\varphi$ (because ${\mathcal L} \varphi$  is also a continuous function). Thus, applying the definition of  the  measure ${\mathcal L}^* \mu$, we re-write the latter equality as $\int  \varphi \, d {\mathcal L}^* \mu_n \rightarrow \int \varphi \, {\mathcal L}^* d \mu$ for any continuous function $\varphi$. In other words ${\mathcal L}^* \mu_n \rightarrow {\mathcal L}^* d \mu$  in the weak$^*$ topology of ${\mathcal M}$ (provided that $\mu_n \rightarrow \mu$). We conclude that    the transfer operator $ {\mathcal L}^*$ is continuous, as wanted.
\end{proof}

\begin{Definition}
\label{definitionL_infty} \em

We denote by $L_{\infty}$ the set of bounded  functions $\varphi: X \mapsto \mathbb{C}$ such that for any probability measure $\mu \in {\mathcal M}$ there exists a measurable function $  \varphi_{\mu}$ that coincides $\mu$-a.e. with $\varphi$.

Thus, for  any $\varphi \in L_{\infty}$ it is well defined the integral of $\varphi$ with respect to any measure $\mu \in {\mathcal M}$, by the following equality
 $\int\varphi \, d \mu := \int  \varphi_{\mu} \, d \mu.$
In particular, it is well defined the following extension of the transfer operator ${\mathcal L}$ to any real function $\varphi \in L_{\infty}$:
  \begin{equation}
  \label{eqn001}
  ({\mathcal L}  \varphi) (x) : = \int \varphi(y) \, P(x, dy), \ \ \forall \ x \in X, \ \ \forall \ \varphi \in L_{\infty},\end{equation}
  where  $P(x, dy)$ is the probability measure constructed by Proposition \ref{proposition-P(x,cdot)} for each $x \in X$.
\end{Definition}

\begin{Proposition}
\label{propositionCalL:L_inftyToL_infty}

For any real function $\varphi \in L_{\infty}$ the function ${\mathcal L} \varphi$ constructed by equality \em (\ref{eqn001}) \em also belongs to $L_{\infty}$. Besides
\begin{equation}
  \label{eqn00}
  \int \varphi \, d{\mathcal L}^* \mu = \int  {\mathcal L} \varphi  \, d \mu \ \ \ \ \forall \ \varphi \in L_{\infty} \ \ \ \ \ \forall \ \mu \in {\mathcal M}.\end{equation}
\end{Proposition}

  \begin{proof}
Since $|\varphi (y)| \leq k $ for all $y \in X$, we have
 $\displaystyle ({\mathcal L} \varphi) (x)  = \int \varphi(y) \, P  (x , dy) \leq k \ \ \ \forall \ x \in X. $
Therefore ${\mathcal L} \varphi$  is bounded.
Let us prove that  ${\mathcal L}  \varphi$ coincides with a measurable function for $\mu$-a.e. $x \in X$.

\noindent{\bf 1st. step. } If $\varphi: X   \mapsto \mathbb{C}$ is continuous, then from hypothesis
 $({\mathcal L} \varphi)  $ is continuous, hence measurable.

\noindent{\bf 2nd. step. } Let us prove that for any open set $V \subset X$, the real function ${\mathcal L} \chi_V $ is measurable.
Since $X$ is a compact metric space, for any open set $V \subset X$ there exists a increasing sequence $\{K_n\}_{n \geq 1}$ of compact sets $K_n \subset X$, such that $\bigcup_{n \geq 1} K_n = V$. So $\lim _{n \rightarrow + \infty} \chi_{K_n}(x) = \chi_V(x)$ for all $x \in X$. From Urysohn Lemma, there exists a sequence of continuous functions $\varphi_n: X \mapsto [0,1]$ such that $\chi_{K_n}(x) \leq \varphi_n(x) \leq \chi_V(x) $ for all $x \in X$. Therefore, \begin{equation}
\label{eqn119a}
\lim_{n\rightarrow + \infty} \varphi_n(x) = \chi_V(x) \ \ \ \forall \ x \in X.\end{equation}
Applying the dominated convergence theorem, we deduce that:
$$ \lim_{n \rightarrow + \infty} \int \varphi_n(y) P (x, dy) = \int \chi_V(y) \, P (x,dy) \ \ \forall \ x \in   X.$$
By the definition of the operator ${\mathcal L}$ we deduce that
\begin{equation}
\label{eqn119}
\lim_{n \rightarrow + \infty} ({\mathcal L}  \varphi_n) (x)  =  ({\mathcal L}   \chi_V) (x) \ \ \forall \ x \in X.\end{equation}
Since $\varphi_n$ is continuous, $({\mathcal L}  \varphi_n)$ is continuous, hence measurable. Besides, the point-wise limit of measurable functions is measurable. We deduce that ${\mathcal L}  \chi_V$ is measurable, as wanted.

\noindent{\bf 3rd. step. }
Let us prove that for any open set $V \subset X$, the following equality holds:
\begin{equation}
\label{eqn118} ({\mathcal L} ^* \mu)(V) = \int ({\mathcal L}  \chi_V) \, d \mu.
\end{equation}
In fact, applying equalities (\ref{eqn000}), (\ref{eqn119a}) and (\ref{eqn119}), and the dominated convergence theorem,  we obtain:
$$\int ({\mathcal L}  \chi_V) \, d   \mu   = \lim_{n \rightarrow + \infty} \int ({\mathcal L} \varphi_n)\, d \mu = \lim_{n \rightarrow + \infty} \int \varphi_n \, d  {\mathcal L} ^* \mu  = \int \chi_V \, d  {\mathcal L}^* \mu  =   ({\mathcal L}^* \mu)(V). $$
So, equality (\ref{eqn118}) is proved.

\noindent{\bf 4th. step. } For any compact set $K \subset X$, equality (\ref{eqn118}) holds, with $K$ in the role of $V$. In fact, $\chi_K = 1 - \chi_V$, where $V = X \setminus K $ is open, hence satisfies equality (\ref{eqn118}). Besides $ {\mathcal L} (1 -\chi_V) = 1 -  {\mathcal L} \chi_V  $ and $({\mathcal L}^* \mu)(K) =  1- ({\mathcal L}^* \mu)(V)$. So equality (\ref{eqn118}) also holds   for $K$ instead of $V$.

\noindent{\bf 5th. step. } Let us prove that for any measurable set $A \subset X$ and any probability measure $\mu$, the function  $ {\mathcal L} \chi_A$ is measurable $\mu$-a.e.  (namely, ${\mathcal L} \chi_A$ coincides with a measurable function up to a set of zero $\mu$-measure).

 Since $(X, {\mathcal A})$ is the measurable space of a compact metric space $X$ with the Borel sigma-algebra ${\mathcal A}$, any probability measure  in $(X, {\mathcal A})$ is regular. So, for any set $A \in {\mathcal A}$ and   any natural number $n \geq 1$, there exists a compact set $K_n \subset A$ and an open set $V_n \supset A$, such that
 $$({\mathcal L}^* \mu)(V_n \setminus K_n) < \frac{1}{n} \ \ \ \forall \ n \geq 1.$$
 It is not restrictive to assume that $K_n \subset K_{n+1} $ and $V_{n+1} \subset V_n$ for all $n \geq 1$.  If not, we substitute $K_n$ by $\bigcup_{j= 1}^n K_j$, and $V_n$ by $\bigcap _{j= 1}^n V_j$.

 Since $K_n \subset A \subset V_n$ we have
 $$\chi_{K_n} \leq \chi_A \leq \chi_{V_n} \ \ \forall \ n \geq 1. $$
 Therefore
 \begin{equation}\label{eqn8a}  \lim _{n \rightarrow + \infty} \chi_{K_n} (x) = \chi_{(\bigcup_{n \geq 1} K_n)} (x) \leq \chi_A (x) \leq \chi_{(\bigcap_{n \geq 1} V_n)}(x) = \lim_{n \rightarrow + \infty} \chi_{V_n} (x) \ \ \ \ \ \forall \ x \in X. \end{equation}
 Thus,  applying the dominated convergence theorem, we deduce that:
 $$\lim_{n \rightarrow + \infty} \int \chi_{K_n} (y) \, P(x, dy) \leq \int \chi_A (y) \, P(x, dy) \leq  \lim_{n \rightarrow + \infty} \int \chi_{V_n} (y)  \, P(x, dy) \ \ \ \forall \ x \in X.$$
 Equivalently:
\begin{equation}
 \label{eqn120} \lim_{n \rightarrow + \infty} ({\mathcal L} \chi_{K_n})(x) \leq  ({\mathcal L} \chi_{A})(x) \leq \lim_{n \rightarrow + \infty} ({\mathcal L} \chi_{V_n})(x)  \ \  \ \ \forall \ x \in X.\end{equation}
 Since $V_n \setminus K_n$ is an open set,   $({\mathcal L} \chi_{V_n \setminus K_n})$ is a measurable function, and   equality (\ref{eqn118}) applies:
 $$0 \leq \int  {\mathcal L} \chi_{V_n \setminus K_n}  \, d \mu  = \int \chi_{V_n \setminus K_n} \, d ({\mathcal L}^* \mu) =  ({\mathcal L}^* \mu) (V_n \setminus K_n) < \frac{1}{n}.$$
 Thus, applying again the dominated convergence theorem, we obtain:
 $$ 0 \leq \int (\lim_{n \rightarrow + \infty } {\mathcal L} \chi_{V_n} - \lim_{n \rightarrow + \infty } {\mathcal L} \chi_{K_n}) \, d \mu  = \lim _{n \rightarrow + \infty} \int  {\mathcal L} \chi_{V_n \setminus K_n}  \, d \mu = 0.$$
 But the integrated function is non negative. Thus it must be null $\mu$-a.e. We have proved that    $$\lim_{n \rightarrow + \infty } {\mathcal L} \chi_{V_n}(x)   =   \lim_{n \rightarrow + \infty } {\mathcal L} \chi_{K_n} (x) \ \ \ \mbox{ for } \mu\mbox{-a.e. } x \in X. $$
 Joining this result with inequalities (\ref{eqn120}) we conclude that $\chi_A (x)$ coincides, for $\mu$- a.e. $x \in X$ with a measurable function. Precisely
 $$({\mathcal L}  \chi_A) (x) = \lim_{n \rightarrow + \infty} ({\mathcal L} \chi_{V_n})(x) = \lim_{n \rightarrow + \infty} ({\mathcal L} \chi_{K_n})(x) \ \ \ \mbox{ for } \mu\mbox{-a.e. } x \in X.$$
 Therefore, taking into account inequalities (\ref{eqn8a}) and that $\chi_{K_n}$ and $\chi_{V_n}$ satisfy equality (\ref{eqn00}), we obtain
 $$\int {\mathcal L} \chi_A \, d \mu = \int \lim_{n \rightarrow + \infty}{\mathcal L} \chi_{V_n} \, d \mu = $$ $$
 \lim_{n \rightarrow + \infty} \int  {\mathcal L} \chi_{V_n} \, d \mu = \lim_{n \rightarrow + \infty} \int \chi_{V_n} \, d {\mathcal L}^* \mu = \int \lim_{n \rightarrow + \infty}   \chi_{V_n} \, d {\mathcal L}^* \mu \geq \int \chi_A \, d {\mathcal L}^* \mu,$$
$$\int {\mathcal L} \chi_A \, d \mu = \int \lim_{n \rightarrow + \infty}{\mathcal L} \chi_{K_n} \, d \mu = $$ $$
 \lim_{n \rightarrow + \infty} \int  {\mathcal L} \chi_{K_n} \, d \mu = \lim_{n \rightarrow + \infty} \int \chi_{K_n} \, d {\mathcal L}^* \mu = \int \lim_{n \rightarrow + \infty}   \chi_{K_n} \, d {\mathcal L}^* \mu \leq \int \chi_A \, d {\mathcal L}^* \mu.$$
 We conclude that $\chi_A$ also satisfies equality (\ref{eqn00}).

\noindent{\bf 6th. step. } Consider a simple measurable function $\varphi$; i.e. $\varphi  $ is a finite linear combination, with real coefficients, of characteristic functions of measurable sets.  Then, the function ${\mathcal L} \varphi$ is a finite linear combination  of $\mu$-a.e. measurable functions, because the operator ${\mathcal L}$ is linear. Since the finite linear combination of measurable functions is measurable, we conclude that ${\mathcal L} \varphi$ coincides   $\mu$-a.e. with a measurable function. Besides, taking into account that the characteristic functions of measurable sets satisfy equality (\ref{eqn00}), by the linearity of the integrals, we deduce that the simple function $\varphi$ also satisfies equality (\ref{eqn00}).

\noindent{\bf 7th. step. } Now, consider any bounded measurable function $\varphi: X \mapsto \mathbb{R}$. It is well known that there exists an increasing (in absolute value)   sequence $\{\varphi_n\}_{n \geq 1}$ of simple measurable functions such that  $\lim _{n \rightarrow + \infty} \varphi_n(x) = \varphi(x)$ for all $x \in X$. Thus,   for all $x \in X$ we can apply the dominated convergence theorem of the integrals of the functions $\varphi_n$ with respect to the   probabilities $P(x, \cdot)$. We deduce:
$$({\mathcal L}\varphi )(x) = \lim_{n \rightarrow + \infty} ({\mathcal L}\varphi_n)(x)  \ \ \ \ \forall \ x \in X \setminus \Delta.  $$
We have already proved that the functions $ {\mathcal L}\varphi_n $ coincide with a measurable function for $\mu$-a.e. point in $X \setminus \Delta$, because $\varphi_n$ is a simple measurable function. So, their point-wise limit also coincides with a measurable function   $\mu$-a.e. Besides, since the simple functions satisfy equality (\ref{eqn00}), by the dominated convergence theorem, the bounded real function $\varphi$ also satisfies it.

\noindent{\bf 8th. step. } Finally, consider any bounded measurable function $\varphi: X \mapsto \mathbb{C}$.  By taking real and imaginary parts of $\varphi$, and taking into account that ${\mathcal L}$ and the integrals are linear, we conclude that ${\mathcal L} \varphi$ coincides $\mu$-a.e. with a measurable function, and satisfies equality (\ref{eqn00}). This ends the proof of Proposition \ref{propositionCalL:L_inftyToL_infty}.
\end{proof}

\section{Proof of the Maximal Ergodic Theorem and its Corollaries} \label{sectionErgodicTheorems}

We will start proving the Maximal Ergodic Theorem
\ref{TheoremMaximalErgodic}. To prove it we need some previous Lemmas:

\begin{Lemma}
\label{lemma22}
For any bounded measurable function $\varphi: X \mapsto \mathbb{R}$, consider the positive and negative parts of $\varphi$  defined by:
$$\varphi^+(x) := \max\{0, \varphi(x)\} \geq 0, \ \ \ \ \varphi^-(x) := -\min\{0, \varphi(x)\} \geq 0, \ \ \ \ \ \varphi= \varphi^+ - \varphi^- \leq \varphi^+.$$
Then $$({\mathcal L} (\varphi ^+)) (x)  \geq ({\mathcal L} \varphi)^+ (x)  \ \ \ \ \forall \ \ x \in X. $$
\end{Lemma}
\begin{proof} On the one hand, applying Proposition \ref{proposition-P(x,cdot)}  we have:
$$({\mathcal L} (\varphi ^+)) (x) = \int \varphi^+(y) \, P(x, dy) \geq \int \varphi (y) \, P(x, dy) = ({\mathcal L} \varphi) (x) \ \ \ \forall \ x \in X.$$
On the other hand, since $\varphi^+ \geq 0$, we have ${\mathcal L} (\varphi ^+) \geq 0 $. So,
 ${\mathcal L} (\varphi ^+) \geq \max \{0, ({\mathcal L} \varphi)\}  = ({\mathcal L} \varphi)^+$,
as wanted.
\end{proof}
\begin{Lemma}
\label{lemmaTeoMaximalErgodic}

Let $\mu \in {\mathcal M}$ such that ${\mathcal L}^* \mu = \mu$. Then, for any bounded measurable function $\varphi : X \mapsto {\mathbb{R}}$:
$$ \int_{\varphi >0} \varphi \, d \mu \geq \int_{({\mathcal L} \varphi) > 0} ({\mathcal L} \varphi) \, d \mu.$$
\end{Lemma}
\begin{proof}
Applying equality (\ref{eqn00}), taking into account that   ${\mathcal L}^* \mu = \mu$, and applying Lemma \ref{lemma22}, we obtain:
   $$ \displaystyle \int_{\varphi > 0} \varphi \, d \mu = \int  \varphi^+ \, d \mu =  \int  \varphi ^+ \, d {\mathcal L}^* \mu = \int  {\mathcal L} (\varphi^+) \, d \mu \geq \int ({\mathcal L}  \varphi)^+  \, d \mu = \int _{{\mathcal L} \varphi >0} {\mathcal L} \varphi \, d \mu.$$
\end{proof}

\subsection
{\bf Proof of Theorem \ref{TheoremMaximalErgodic}.}

\begin{proof}
The sequence $\{\varphi_n\}_{n \geq 1}$ is non decreasing. Thus, for all $n \geq 1$, the set $E_n := \{x \in X: \varphi_n> 0\}$ is contained in $E_{n+1}$. Since $E (\varphi) = \bigcup_{n \geq 1} E_n$, we obtain:
   $ \displaystyle \int_{E(\varphi)} \varphi \, d \mu = \lim_{n \rightarrow + \infty} \int_{E_n} \varphi \, d \mu.$
  So, to  prove inequality (\ref{eqnTeoErgMaximal}), it is enough to prove the following inequality:
  \begin{equation}
 \label{eqn11}
 I_n:= \int_{E_n} \varphi \, d \mu \geq 0  \ \ \forall \ n \geq 1 \ \ \ \  \ \ \ \mbox{(to be proved)}.\end{equation}
We have
\begin{equation}
\label{eqn12}
I_n = \int_{\varphi_n > 0} \varphi \, d \mu = \int _{\varphi_n >0, \ {\mathcal L} \varphi_n \leq 0} \varphi \, d \mu + \int_{\varphi_n >0, \ {\mathcal L} \varphi_n > 0} \varphi \, d \mu. \end{equation}
For all $j \geq 1$, denote $\psi_j := \varphi + {\mathcal L} \varphi + \ldots + {\mathcal L}^{j-1} \varphi$. We assert that \begin{equation}\label{eqn13}{\mathcal L}(\max_{1 \leq j \leq n} \psi_j) \geq \max_{ 1 \leq j \leq n} {\mathcal L} \psi_j.\end{equation}
In fact, $\max_{1 \leq j \leq n} \psi_j \geq \psi_i $ for all $1 \leq i \leq n$. Thus
$$ ({\mathcal L}(\max_{1 \leq j \leq n} \psi_j))(x)= \int  \max_{1 \leq j \leq n} \psi_j (y) P(x, dy) \geq \int \psi_i (y) P(x, dy)= ({\mathcal L} \psi_i) (x) \ \ \forall \ x \in X \ \ \forall \ i= 1, \ldots, n. $$
Now, inequality (\ref{eqn13}) is proved.

Let us compute both integrals at right in equality (\ref{eqn12}):
$$ {\mathcal L} \varphi_n =  {\mathcal L} (\max_{1 \leq j \leq n} \psi_j) \geq \max_{ 1\leq j \leq n} {\mathcal L} \psi_j = \max \{{\mathcal L} \varphi, {\mathcal L} \varphi + {\mathcal L}^2 \varphi, \ldots, {\mathcal L} \varphi + {\mathcal L}^2 \varphi + \ldots +  {\mathcal L}^n \varphi\}.  $$
Therefore,   ${\mathcal L} \varphi_n \leq 0$ implies $\max \{{\mathcal L} \varphi, {\mathcal L} \varphi + {\mathcal L}^2 \varphi, \ldots, {\mathcal L} \varphi + {\mathcal L}^2  + \ldots +  {\mathcal L}^n \varphi\} \leq 0$, hence:
$$({\mathcal L} \varphi_n)(x) \leq 0 \ \ \Rightarrow \ \ \varphi_n (x) = \max\{ \varphi (x), (\varphi + {\mathcal L} \varphi) (x),  \ldots, (\varphi {\mathcal L} \varphi + \ldots +  {\mathcal L}^{n-1} \varphi)(x)\} = \varphi (x).$$
Thus, the first integral at right in equality (\ref{eqn12})  can be written as follows:
\begin{equation}
\label{eqn12-1} \int_{\varphi_n > 0, \ {\mathcal L}\varphi_n \leq 0} \varphi \, d \mu = \int_{\varphi_n > 0, \ {\mathcal L}\varphi_n \leq 0} \varphi_n \, d \mu.
\end{equation}
Now, let us compute the second integral at right in equality (\ref{eqn12}). Applying inequality (\ref{eqn13})  we obtain:
$$\varphi + {\mathcal L} \varphi_n = \varphi + {\mathcal L} (\max_{1 \leq  j\leq n} \psi_j) \geq
 \varphi + \max_{1 \leq j \leq n} ({\mathcal L} \psi_j) = $$
 $$\varphi + \max \{{\mathcal L}\varphi, {\mathcal L}\varphi +{\mathcal L}^2\varphi, \ldots, {\mathcal L}\varphi + {\mathcal L}^2\varphi + \ldots +{\mathcal L}^n\varphi\} = $$ \begin{equation}
 \label{eqn13a} \max   \{\varphi + {\mathcal L}\varphi, \varphi + {\mathcal L}\varphi +{\mathcal L}^2\varphi, \ldots, \varphi + {\mathcal L}\varphi + {\mathcal L}^2\varphi + \ldots +{\mathcal L}^n\varphi\}.
 \end{equation}
Besides
\begin{equation}
 \label{eqn13b}{\mathcal L} \varphi_n > 0 \ \ \Rightarrow \ \ \varphi + {\mathcal L} \varphi_n > \varphi.\end{equation}
Joining inequalities (\ref{eqn13a}) and (\ref{eqn13b}), we deduce:
$${\mathcal L}\varphi_n > 0 \ \ \Rightarrow \ \ \varphi + {\mathcal L} \varphi_n  \geq \max   \{\varphi, \ \varphi + {\mathcal L}\varphi,   \ldots, \varphi + {\mathcal L}\varphi   + \ldots +{\mathcal L}^n\varphi\} \geq$$
$$\max \{\varphi, \ \varphi + {\mathcal L}\varphi,   \ldots, \varphi + {\mathcal L}\varphi +  \ldots +{\mathcal L}^{n-1}\varphi\} = \varphi_n. $$
In brief, we have proved that
$${\mathcal L}\varphi_n > 0 \ \ \Rightarrow \ \ \varphi + {\mathcal L} \varphi_n  \geq  \varphi_n \ \ \Rightarrow \ \ \varphi \geq \varphi_n - {\mathcal L} \varphi_n. $$
Substituting the latter inequality in the second integral at right of equality (\ref{eqn12}), we obtain:
\begin{equation}
\label{eqn12-2} \int_{\varphi_n > 0, \ {\mathcal L}\varphi_n  > 0} \varphi \, d \mu \geq  \int_{\varphi_n > 0, \ {\mathcal L}\varphi_n  > 0} \varphi_n \, d \mu - \int_{\varphi_n > 0, \ {\mathcal L}\varphi_n  > 0} {\mathcal L} \varphi_n \, d \mu. \end{equation}
Now, we use equality (\ref{eqn12-1}) and inequality (\ref{eqn12-2}) to obtain a lower bound of the integral (\ref{eqn12}):
$$I_n \geq \int_{\varphi_n > 0, {\mathcal L} \varphi_n \leq 0} \varphi_n \, d \mu + \int_{\varphi_n > 0, {\mathcal L} \varphi_n > 0} \varphi_n \, d \mu  - \int_{\varphi_n > 0, {\mathcal L} \varphi_n > 0} ({\mathcal L} \varphi_n) \, d \mu   =   \int_{\varphi_n > 0} \varphi_n \, d \mu    - \int_{\varphi_n > 0, {\mathcal L} \varphi_n > 0} ({\mathcal L} \varphi_n) \, d \mu.  $$
Since $\{ \varphi_n > 0, \ {\mathcal L} \varphi_n > 0 \} \subset \{  {\mathcal L} \varphi_n > 0 \} $ and the function ${\mathcal L}\varphi_n$ is positive on those sets, we obtain:
$$I_n \geq \int_{\varphi_n > 0 } \varphi_n \, d \mu -   \int_{{\mathcal L} \varphi_n > 0} ({\mathcal L} \varphi_n) \, d \mu. $$
Finally, applying Lemma \ref{lemmaTeoMaximalErgodic}, the difference at right in the latter inequality is non negative. We conclude that
 $I_n \geq 0$, proving assertion (\ref{eqn11}) as wanted, and ending the proof of Theorem \ref{TheoremMaximalErgodic}.
\end{proof}

\subsection{Proof of Corollary \ref{corollaryMaximalErgTheo}.}

To prove Corollary \ref{corollaryMaximalErgTheo}, we need a previous lemma:

\begin{Lemma}
\label{lemmaForCorollaryMaximalErgTheo}
Let $\mu \in {\mathcal M}$ ($\mu$ is not necessarily fixed by ${\mathcal L}^*$). Let $A \subset X$ be a measurable set that is $\mu$-a.e. ${\mathcal L}$-invariant.
Then, for any measurable bounded function $\varphi: X \mapsto \mathbb{C}$ the following equality holds:
$$({\mathcal L} (\chi_A \cdot \varphi)) (x) = \chi_A (x) \cdot ({\mathcal L} \varphi) (x) \ \ \ \mu\mbox{-a.e. } x \in X.$$
\end{Lemma}
\begin{proof} For $\mu$-a.e. $x \in X$ we have $({\mathcal L} \chi_A )(x) = P (x,A) = \chi_A(x)$. Therefore
$$P(x,A) = 1 \ \ \mbox {if }\ \  x \in A; \ \ \ \ P(x,A) = 0 \mbox \ \ \mbox{if }\ \  x \not \in A, \ \ \ \ \ \ \mu\mbox{-a.e. } x \in X.$$
$$x \in A \ \ \Rightarrow  \ \ \chi_A (x) \cdot ({\mathcal L} \varphi) (x) = {(\mathcal L} \varphi)(x) = \int     \varphi(y) \ P (x, dy) =  \int_A \varphi(y) \ P(x, dy) + \int _{X \setminus A} \varphi (y) \ P(x, dy).$$ But for $\mu$-a.e. $x \in A$ we have $P(x,X \setminus A) = 0$. So, the integral at right in the above equality is zero. We obtain:  $$\mbox{for } \mu\mbox{-a.e. }x \in A,\ \ \chi_A (x) \cdot ({\mathcal L} \varphi) (x) = \int_A \varphi(y) \ P(x, dy) = \int \chi_A(y) \cdot \varphi(y) \ P(x, dy) = ({\mathcal L} (\chi_A \varphi))(x).$$
We have proved   Lemma \ref{lemmaForCorollaryMaximalErgTheo} for $\mu$-a.e. $x \in A$. Now, let us consider $x \not \in A$:
$$x \not \in A \ \ \Rightarrow  \ \ \chi_A (x) \cdot ({\mathcal L} \varphi) (x) =  0. $$ Besides, for $\mu$-a.e. $x \not \in A$ we have $P(x,A) = 0$. We obtain:  $$\mbox{for } \mu\mbox{-a.e. }x \not \in A, \ \ \chi_A (x) \cdot ({\mathcal L} \varphi) (x) =  0 =  \int _A \varphi(y) \ P(x, dy) = \int \chi_A(y) \cdot \varphi(y) \ P(x, dy)= ({\mathcal L} (\chi_A \varphi))(x). $$
ending the proof of Lemma \ref{lemmaForCorollaryMaximalErgTheo}. \end{proof}

\noindent {\bf End of the Proof of Corollary \ref{corollaryMaximalErgTheo}.}

\begin{proof}
First, let us prove Corollary \ref{corollaryMaximalErgTheo} in the particular case $\alpha= 0$.
We consider the measurable real function $$g_n:= \chi_A \cdot \varphi_n = \sum_{j= 0}^{n-1} \chi_A ({\mathcal L}^j \varphi).$$
Applying Lemma \ref{lemmaForCorollaryMaximalErgTheo}, we obtain:
$$g_n = \sum_{j= 0}^{n-1}    {\mathcal L}^j (\chi_A \cdot \varphi) \ \ \ \mu\mbox{-a.e.}$$

By construction of $g_n$, if $x \not \in A$ then $g_n(x) = 0$ for all $n \geq 1$. Therefore, if $\sup_{n \geq 1} g_n(x) > 0$ then $x \in A$. Conversely, by  hypothesis $A \subset C_0$, hence $\sup_{n \geq 1} g_n(x) > 0$ if $x \in A$. We have proved that
$$A = \{ x \in X: \ \ \sup_{n \geq 1} g_n(x) > 0\}.$$
Applying Theorem \ref{TheoremMaximalErgodic} we obtain:
$$ \int_A \chi_A \, \varphi \, d \mu  \geq 0; \ \ \ \ \int_A \, \varphi \, d \mu \geq 0 = \alpha \, \cdot \mu(A) \ \ \ \mbox{ if } \ \ \alpha = 0.$$
We have proved Corollary \ref{corollaryMaximalErgTheo} in the particular case   $\alpha = 0$.
Now, let us prove it for any real value of $\alpha$.
Consider the function
 $h_n := \varphi_n - n \cdot \alpha.$
Since ${\mathcal L} \alpha = \alpha$ and ${\mathcal L}$ is linear, we obtain:
$$h_n = \sum_{j=0}^{n-1} {\mathcal L}^j(\varphi- \alpha).$$
Besides $$   C_{\alpha}(\varphi) := \{ x \in X \colon \sup_{n \geq 1} \frac{\varphi_n (x)}{n} > \alpha\} = \{x \in X \colon \sup_{n \geq 1} \frac{h_n (x)}{n} > 0\} = C_0(h_1). $$
Thus, applying to the measurable function $h_1$ the result proved in the case $\alpha  = 0$, we conclude that
$$ \int _A \ (\varphi - \alpha )\, d \mu \geq 0; \ \ \mbox{ hence } \int _A \varphi \geq \alpha \cdot \mu(A).$$
\end{proof}

\subsection{Proof of Corollary \ref{corollaryMaximalErgTheo2}}

\begin{proof}
We apply Corollary \ref{corollaryMaximalErgTheo} to the function $-\varphi$, with $ - \beta$ instead of $\alpha$:
$$\int_A - \varphi \, d \mu \geq - \beta \mu(A), \ \ \mbox{ hence } \ \ \int_A  \varphi \, d \mu \leq   \beta \mu(A).$$
\end{proof}

%

\section{Kakutani's Ergodic Theorem}
 The purpose of this section is to give a  proof of the following version of Kakutani's Ergodic Theorem, applied to measures that are stationary (i.e. invariant under the transfer operator ${\mathcal L}^*$),  using the Maximal Ergodic Theorem  \ref{TheoremMaximalErgodic} that we have already proved in Section \ref{sectionErgodicTheorems}.
 
\begin{Theorem} {\bf (Kakutani's Ergodic Theorem for ${\mathcal L}^*$-invariant measures)} \label{TheoremMainErgodic}

If ${\mathcal L}^* \mu = \mu$, then for any $\varphi \in L_{\infty}$
there exists
$$\widetilde \varphi(x) = \lim_{n \rightarrow + \infty} \frac{1}{n} \sum_{j= 0}^{n-1} ({\mathcal L}^j \varphi) (x) \ \ \ \mu\mbox{-a.e. }x \in X.$$
\end{Theorem}

Before proving Theorem \ref{TheoremMainErgodic} we will prove a lemmata:
\begin{Lemma} \label{lemmaInvariantFunctions}
Let $\varphi: X \mapsto \mathbb{R}$ be a bounded measurable function, and let $\mu \in {\mathcal M}$ such that ${\mathcal L}^* \mu = \mu$.
Assume that $\varphi$ is ${\mathcal L}$-invariant $\mu$-a.e.; precisely
$$({\mathcal L} \varphi) (x) = \varphi (x) \ \ \mu\mbox{-a.e. } x \in X.$$
Then, for any real number $\alpha $  the set $$A_{\alpha} = \{x \in X \colon \varphi(x) \geq \alpha\} $$ is ${\mathcal L}$-invariant $\mu$-a.e.; namely,
$$\chi_{A_{\alpha}} (x) = ({\mathcal L} \chi_{A_{\alpha}}) (x) =  P (x, A_{\alpha}) \ \ \ \mu\mbox{-a.e. } x \in X.$$
\end{Lemma}
\begin{proof}
By hypothesis $({\mathcal L} \varphi)  (x) = \varphi (x) $ for $\mu$-a.e. $x \in X$. Thus, applying Lemma \ref{lemma22}:
$$ ({\mathcal L} (\varphi^+))(x) \geq ({\mathcal L}  \varphi) ^+ (x) =   \varphi^+ (x) \ \ \ \mu\mbox{-a.e. }   x \in X. $$
Besides, applying equality (\ref{eqn00}) and taking into account that $\mu$ is ${\mathcal L}^*$-invariant, we obtain
$$\int \Big ( {\mathcal L} (\varphi^+)  -  \varphi^+  \Big) \, d \mu = \int \varphi^+ \, d {\mathcal L}^* \mu - \int \varphi^+ \, d \mu = \int \varphi^+ \, d   \mu - \int \varphi^+ \, d \mu = 0.$$
But  the integrated function ${\mathcal L} (\varphi^+)  -  \varphi^+$  is non negative. So it is zero $\mu$-a.e. We have proved that
$$({\mathcal L} \varphi^+)(x)  =   \varphi^+ (x)  \ \ \mbox{ for } \mu\mbox{-a.e. } x \in X.$$
Since $\varphi^+(x)= \chi_{A_0}(x) \cdot \varphi(x)$  for all $x \in X$, we obtain:
\begin{equation}
\label{eqn15} \chi_{A_0}(x) \cdot \varphi(x) = \varphi^+ (x) = ({\mathcal L} (  \varphi^+)) (x) = \int  \varphi^+(y) \ P(x, dy)  \  \ \   \ \mu\mbox{-a.e. } x \in X.\end{equation}
For all $x \in A_0$, we have $\chi_{A_0}(x) = 1$. Therefore, from equality (\ref{eqn15}) we deduce: \begin{equation}
\label{eqn14}
\mbox{For } \mu\mbox{-a.e. } x \in A_0, \ \ \ \varphi(x) = \int  \varphi^+(y) \ P(x, dy) \geq \int  \varphi (y) \ P(x, dy),\end{equation}
where the inequality  at right  is an equality   only if $\varphi^+ (y) = \varphi (y)$ for $P(x, \cdot)$-a.e. $y \in X$. This latter assertion occurs  only if $P( x , A_0) = 1$. By hypothesis,   $\varphi(x) = ({\mathcal L} \varphi) (x) = \int \varphi(y) P(x, dy)$. So, the inequality at right in (\ref{eqn14}) is indeed an equality. We have proved that $P(x, A_0) = 1$ for $\mu$-a.e. $x \in A_0$. In other words:
$$\chi_{A_0}(x) = 1 \ \ \ \Rightarrow \ \ \ P (x, A_0) = 1 \ \ \ \mbox{for } { \mu}\mbox{-a.e. } x \in X.$$
Therefore $P(x, A_0) \geq \chi_{A_0}(x)$ for $\mu$-a.e. $x \in X$. But
$$\int (  P (x, A_0) - \chi_{A_0}(x) )\, d \mu(x) = \int ({\mathcal L} \chi_{A_0})(x) \, d \mu(x) - \int \chi_{A_0} (x)\, d \mu(x) = \int \chi_{A_0} d {\mathcal L}^* \mu - \int \chi_{A_0} \, d \mu = 0.$$
So, we deduce that
$$  P (x, A_0) =  \chi_{A_0}(x) \ \ \ \mbox{for } { \mu}\mbox{-a.e. } x \in X,$$
ending the proof of Lemma \ref{lemmaInvariantFunctions} in the case $\alpha= 0$.

Now, let us consider any real value of $\alpha$. Note that
$$ A_{\alpha} (\varphi):= \{ x \in X \colon \varphi \geq \alpha\}= A_0(\varphi- \alpha).$$
Since $\varphi$ and the constant $\alpha$ are $\mu$-a.e. ${\mathcal L}$-invariant, also the function $\varphi - \alpha$ is $\mu$-a.e. ${\mathcal L}$-invariant. So, applying the  case above to the function $\varphi- \alpha$, we deduce that the set $A_0(\varphi- \alpha) = A_{\alpha}(\varphi)$ is $\mu$-a.e. ${\mathcal L}$-invariant, as wanted.
\end{proof}

\begin{Lemma}
 \label{lemmaInvariantFunctions2}

 Let $\varphi: X \mapsto \mathbb{R}$ be a bounded measurable function, and let $\mu \in {\mathcal M}$ such that ${\mathcal L}^* \mu = \mu$.
Assume that $\varphi$ is ${\mathcal L}$-invariant $\mu$-a.e.; precisely
$$({\mathcal L} \varphi) (x) = \varphi (x) \ \ \mu\mbox{-a.e. } x \in X.$$
Then, for any pair of real numbers $\alpha $  and $\beta$, the sets $$C_{\alpha} = \{x \in X \colon \varphi(x) > \alpha\}, \ \ \ \mbox{\em  and } \ \ \ B_{\beta} = \{x \in X \colon \varphi(x) < \beta\} $$ are ${\mathcal L}$-invariant $\mu$-a.e.; namely,
$$\chi_{C_{\alpha}} (x) = ({\mathcal L} \chi_{C_{\alpha}}) (x) =  P (x, C_{\alpha})   \ \ \mbox{ and } \ \
 \chi_{B_{\beta}} (x) = ({\mathcal L} \chi_{B_{\beta}}) (x) =  P (x, B_{\beta}) \ \ \ \mu\mbox{-a.e. } x \in X,$$
\end{Lemma}

\begin{proof}
On the one hand, applying Lemma \ref{lemmaInvariantFunctions} we know  that the set
 $ \{x \in X \colon \  \varphi (x) \geq \beta\} $  is ${\mathcal L}$-invariant $\mu$-a.e. Hence, its complement
$B_{\beta}$
is also ${\mathcal L}$-invariant $\mu$-a.e. In other words
$$\chi_{B_{\beta}} = P(X, B_{\beta}) \ \ \mu\mbox{-a.e. } x \in X. $$
On the other hand, for all $n \geq 1$ we  can apply Lemma \ref{lemmaInvariantFunctions} to the set
$$E_{\alpha + (1/n)} := \{x \in X \colon \   \varphi (x) \geq \alpha + (1/n)\}.  $$
We deduce that $E_{\alpha + (1/n)} $ is ${\mathcal L}$-invariant $\mu$-a.e. Thus, \begin{equation}
\label{eqn16}
\chi_{E_{\alpha + (1/n)}} (x) = P(x, E_{\alpha + (1/n)}) \ \ \mu\mbox{-a.e. } x \in X. \end{equation}
Besides,
 $C_{\alpha}  =  \bigcup_{n \geq 1} E_{\alpha + (1/n)}. $
Thus,   $\chi_{C_{\alpha}} (x)  = \lim_{n \rightarrow + \infty} \chi_{E_{\alpha + (1/n)}}(x)$ for all $ x \in X,$
and  by the dominated convergence theorem, we deduce that
 $P(x, C_{\alpha }) = \lim_{n \rightarrow + \infty} P(x, E_{\alpha + (1/n)})  \ \ \ \forall  \ x \in X.$
Finally, taking $n \rightarrow + \infty$ in equality (\ref{eqn16}) with $x $ fixed, we obtain
 $\chi_{C_{\alpha}} (x) = \lim_{n \rightarrow + \infty} P(x, E_{\alpha + (1/n)}) \ \ \mu\mbox{-a.e. } x \in X,$  concluding that $\chi_{C_{\alpha}} (x) = P(x, C_{\alpha})$  for $\mu\mbox{-a.e. } x \in X,$
as wanted.
\end{proof}
\begin{Lemma}
\label{lemmaMainErgodicTheorem}
Let $\varphi: X \mapsto \mathbb{R}$ be a bounded measurable function, and let $\alpha, \beta$ be real numbers.
Construct the set
$$A  = \{x \in X \colon \ \ \ \limsup_{n \rightarrow + \infty} \frac{1}{n} \sum_{j= 0}^{n-1} ({\mathcal L}^j \varphi)(x) > \alpha, \ \ \ \ \liminf_{n \rightarrow + \infty} \frac{1}{n} \sum_{j= 0}^{n-1} ({\mathcal L}^j \varphi)(x) < \beta\}.$$
Then, for any measure $\mu \in {\mathcal M}$ that is fixed by the transfer operator ${\mathcal L}^*$, the set $A $ is ${\mathcal L}$-invariant $\mu$-a.e.

\end{Lemma}

\begin{proof} It is standard to check that the following real functions $$
 \psi_1 (x)  := \liminf_{n \rightarrow + \infty} \frac{1}{n} \sum_{j= 0}^{n-1} ({\mathcal L}^j \varphi)(x), \ \ \ \ \ \ \
\psi_2 (x)  :=  \limsup_{n \rightarrow + \infty} \frac{1}{n} \sum_{j= 0}^{n-1} ({\mathcal L}^j \varphi)(x),$$
are $ {\mathcal L} $-invariant. Thus, applying Lemma \ref{lemmaInvariantFunctions2} the sets
$$C_{\alpha} := \{x \in X \colon \ \ \psi_2 > \alpha\},  \ \ \ \ \
 B_{\beta} := \{x \in X \colon \ \ \psi_1 < \beta\},$$
satisfy
$$\chi_{C_{\alpha}}(x) = P(x, C_{\alpha}) \ \ \ \mu\mbox{-a.e. } x \in X, \ \ \ \ \ \ \ \ \ \chi_{B_{\beta}}(x) = P(x, B_{\beta}) \ \ \ \mu\mbox{-a.e. } x \in X.$$
On the one hand, for $\mu$-a.e.   $x \in C_{\alpha} \cap B_{\beta}$, we have     $\chi_{C_{\alpha}} = P (x, C_{\alpha}) = 1$  and  $\chi_{B_{\beta}} = P (x, B_{\beta}) = 1$. Since the  intersection of two sets of probability  1 also has probability 1, we deduce that $$P(x, C_{\alpha} \cap B_{\beta}) = 1 \ \ \mbox{ for } \mu\mbox{-a.e. } x \in C_{\alpha} \cap B_{\beta}.$$
On the other hand, for $\mu$-a.e.   $x \not \in C_{\alpha} \cap B_{\beta}$, we have     $\chi_{C_{\alpha}} = P (x, C_{\alpha}) = 0$  or $\chi_{B_{\beta}} = P (x, B_{\beta}) = 0$. The intersection of two sets, when at least one of them has zero probability, also has zero probability. We deduce that
$$P(x, C_{\alpha} \cap B_{\beta}) = 0 \ \ \mbox{ for } \mu\mbox{-a.e. }  x \not \in C_{\alpha} \cap B_{\beta}.$$
Finally, observe that $A  = C_{\alpha} \cap B_{\beta}$. We conclude that
 $\chi_{A } (x) = P(x, A ) \ \ \mbox{ for } \mu\mbox{-a.e. }  x \in X,$
ending the proof of Lemma \ref{lemmaMainErgodicTheorem}.
\end{proof}

\subsection
{\bf Proof of Theorem \ref{TheoremMainErgodic}.}
 
\begin{proof} Due to the linearity of the transfer operator ${\mathcal L}$, it is enough to prove Theorem \ref{TheoremMainErgodic} for real functions $\varphi \in L_{\infty}$. Denote   $ \displaystyle \varphi_n  :=  \sum_{j= 0}^{n-1} {\mathcal L}^j \varphi.$
For any pair of rational numbers $\alpha$ and $ \beta$  such that $\alpha > \beta$, we construct the set
$$A_{\alpha, \beta}= \{x \in X \colon \   \limsup_{n \rightarrow + \infty} \frac{1}{n}  \varphi_n (x)  >   \alpha, \ \  \liminf_{n \rightarrow + \infty} \frac{1}{n}  \varphi_n (x) < \beta\}.$$
Applying Lemma \ref{lemmaMainErgodicTheorem}, the set $A_{\alpha, \beta}$ is ${\mathcal L}$-invariant $\mu$-a.e.
Besides, if $x \in A_{\alpha, \beta}$, then $\sup_{n \geq 1}  { \varphi_n (x)}/{n} > \alpha $ and  $\inf_{n \geq 1} \varphi_n (x)/n  < \beta $. Thus, applying Corollaries \ref{corollaryMaximalErgTheo} and \ref{corollaryMaximalErgTheo2}, we obtain:
$$\alpha  \cdot \mu (A_{\alpha, \beta}) \leq  \int_{A_{\alpha, \beta}} \varphi \leq \beta \cdot \mu (A_{\alpha, \beta}). $$
Since $\alpha > \beta$, we deduce that
 $\mu(A_{\alpha, \beta}) = 0.$
The set of all the pair of rational numbers $\alpha$ and $ \beta$ such that $\alpha > \beta$ is countably infinite. Thus,
 $ \displaystyle \mu \ \Big( {\bigcup}_{  \alpha, \beta \in {\mathbb{Q}}, \   { \alpha > \beta}} \ \ \  {A_{  \alpha, \beta} } \Big)  = 0.$
Finally, we observe that
$${\bigcup}_{   \alpha, \beta \in {\mathbb{Q}}, \   { \alpha > \beta}} \ \ \  {A_{  \alpha, \beta} }  = \Big \{x \in X \colon \ \ \lim_{n \rightarrow +\infty} \frac{1}{n} \varphi_n (x) \mbox{ does not exists}\Big\}. $$
to conclude that
 $\displaystyle \lim_{n \rightarrow + \infty} \frac{1}{n} \varphi_n (x) \ \ \mbox{ exists } \mu\mbox{-a.e. } x \in X,$
ending the proof of Theorem \ref{TheoremMainErgodic}.
\end{proof}

\section{Ergodic Measures} \label{sectionErgodicMeasures}

\begin{Proposition}
\label{propositionErgodMeasuresInvariantFunctions}
A probability measure $\mu$ that is fixed by the operator ${\mathcal L}^*$ is ergodic   if and only if   any function $\varphi \in L_{\infty}$ that is $\mu$-a.e. ${\mathcal L}$-invariant   is constant $\mu$-a.e.
\end{Proposition}
\begin{proof} First, assume that any function $\varphi \in L_{\infty}$ that is $\mu$-a.e. ${\mathcal L}$-invariant is constant $\mu$-a.e. Let us prove that $\mu$ is ergodic according to Definition \ref{definitionErgodic}. Consider a $\mu$-a.e. $ {\mathcal L}$-invariant  set  $A \subset M$, Equivalently, $\chi_A$ is a $\mu$-a.e. ${\mathcal L}$-invariant function. Thus, it is constant $\mu$-a.e. Since $\chi_A$ can take only the values 1 or 0, we deduce that either $\chi_A(x) = 0$ for $\mu$-a.e. $x \in X$, or $\chi_A(x) = 1$ for $\mu$-a.e. $x \in X$. In other words, either $\mu(A) = 0$ or $\mu(A) = 1$, proving that $\mu$ is ergodic.

 Conversely, assume that $\mu$ is ergodic according to Definition \ref{definitionErgodic}, and consider any function $\varphi \in L_{\infty}$  that is ${\mathcal L}$-invariant $\mu$-a.e. Denote
 $$A_{\alpha} := \{x \in X \colon \ \varphi(x) \geq \alpha\}. $$
 Applying Lemma \ref{lemmaInvariantFunctions} the set $A_{\alpha}$ is ${\mathcal L}$-invariant $\mu$-a.e. So, from Definition \ref{definitionErgodic}, we deduce that $\mu(A_{\alpha}) \in \{0,1\}$ for all $\alpha \in \mathbb{R}$. By construction $\mu(A_{\alpha})$ is non increasing with $\alpha$, and it is zero for all values of $\alpha$ large enough (because $\varphi$ is bounded). Thus, there exists
 $$k := \sup \{\alpha \in \mathbb{R}: \   \mu(A_{\alpha}) = 1\} \in \mathbb{R}.$$
 We deduce that $\mu(A_{k + \epsilon}) = 0$ and $ \mu(A_{k - \epsilon}) = 1$ for all $\epsilon >0$. So
 $$\mu\Big((X \setminus A_{k + \epsilon}) \bigcap A_{k - \epsilon}\Big) = 1 \ \ \forall \ \epsilon >0.$$
 In other words: $$\mu\Big( \Big\{ x \in X \colon \ k - \epsilon \leq \varphi (x)  < k + \epsilon \Big\}\Big) = 1 \ \ \ \forall \ \epsilon >0. $$ In particular,
 $$\mu\Big( \Big\{ x \in X \colon \ k - \frac{1}{n} \leq \varphi (x)  < k + \frac{1}{n} \Big\}\Big) = 1 \ \ \ \forall \ n \geq 1. $$
 Then, $$ \mu\Big( \bigcap_{n \geq 1} \Big\{ x \in X \colon \ k - \frac{1}{n} \leq \varphi (x)  < k + \frac{1}{n} \Big\}\Big) = 1,$$
 or equivalently $\varphi(x) = k$ for $\mu$-a.e. $x \in X$. We have proved Proposition \ref{propositionErgodMeasuresInvariantFunctions}.
\end{proof}

\begin{Corollary} \label{CorollaryMainErgodicTheorem}
{\bf (Ergodic Theorem for ${\mathcal L}^*$ ergodic measures)}

If $\mu$ is ergodic for the operator ${\mathcal L}^*$, then  for all $\varphi \in L_{\infty}$:
$$\lim_{n \rightarrow + \infty } \frac{1}{n} \sum_{j= 0}^{n-1} ({\mathcal L}^j \varphi)(x) = \int \varphi \, d \mu \ \ \ \mu\mbox{-a.e. } x \in X.$$
\end{Corollary}
\begin{proof}
Applying Theorem \ref{TheoremMainErgodic} the above limit exists for $\mu$-a.e. $x \in X$. Denote it by $\widetilde \varphi (x)$. Since the function $\widetilde \varphi$ is ${\mathcal L}$-invariant $\mu$-a.e., and the measure $\mu$ is ergodic, we apply Proposition \ref{propositionErgodMeasuresInvariantFunctions} to deduce that there exists a constant $k$ such that $\widetilde \varphi (x) = k$ for $\mu$-a.e. $x \in X$. Now, it is enough to check that $k= \int \varphi \, d \mu$. In fact, by the dominated convergence theorem, we have
$$\int \widetilde \varphi \, d \mu = \int \lim _{ n \rightarrow + \infty} \frac{1}{n} \sum_{j = 0}^{n-1}  ({\mathcal L}^j  \varphi)(x) \, d \mu(x) =  \lim _{ n \rightarrow + \infty} \frac{1}{n} \sum_{j = 0}^{n-1} \int {\mathcal L}^j  \varphi \, d \mu. $$
But taking into account that $\mu$ is ${\mathcal L}^*$-invariant, we have $$\int {\mathcal L}^j \varphi \, d \mu  = \int \varphi \, d {{\mathcal L}^*}^j \mu = \int \varphi \, d \mu. $$
Therefore,
$$ \int \widetilde \varphi \, d \mu = \lim _{ n \rightarrow + \infty} \frac{1}{n} \sum_{j = 0}^{n-1} \int   \varphi \, d \mu = \int \varphi \, d \mu. $$
Finally, since  $\widetilde \varphi = k $ for $\mu$-a.e. $x \in K$, the above equality implies that $k = \int \varphi \, d \mu$, as wanted.
\end{proof}

\subsection
{\bf End of the Proof of Theorem \ref{MainTheorem1}}

\begin{proof}
For each fixed natural value of $p \geq 1 $, the  operator ${\mathcal L}^p$ transforms continuous functions into continuous functions. Besides, it is positive, bounded by 1, and ${\mathcal L}^p (1) = 1$. So, ${\mathcal L}^p$ satisfies Definition \ref{DefinitionTransferOperatorL} and is a transfer operator. Besides, after applying equality (\ref{eqn000}) $p$ times, we deduce that the dual transfer operator of ${\mathcal L}^p$ is ${{\mathcal L}^*}^p$. So, we can apply all the results proved along this paper to ${\mathcal L}^p $ instead of ${\mathcal L}$, and to ${{\mathcal L}^*}^p  $ instead of ${\mathcal L}^*$.
In particular,  Theorem \ref{MainTheorem1} follows from Theorem \ref{TheoremMainErgodic} and Corollary \ref{CorollaryMainErgodicTheorem}, using  ${\mathcal L}^p $ instead of ${\mathcal L}$, and  ${{\mathcal L}^*}^p  $ instead of ${\mathcal L}^*$.
\end{proof}

\section{Conclusions and Further Research} 

We have proved the Maximal Ergodic Theorem \ref{TheoremMaximalErgodic} and the Pointwise Ergodic Theorem \ref{MainTheorem1} of Periodic Measures, for the transfer operator that is associated to a Markovian stochastic dynamical system, obtained by adding noise with any probability distribution  at each iterate of a deterministic continuous system. As a consequence, we have also proved Corollaries \ref{corollaryMaximalErgTheo} and  \ref{corollaryMaximalErgTheo2}, which provide different statements of the  Maximal Ergodic Theorem for noisy systems. Besides, we have proved Kakutani's Ergodic Theorem \ref{TheoremMainErgodic}, also  as a consequence of them.

The relevance of the ergodic theorems for the transfer operators that are proved along this paper, is the extension they provide to stochastic markovian processes, of the classical pointwise ergodic theorems for deterministic systems. They hold in particular for noisy systems, i.e. for dynamical systems which are randomly perturbed by noise, independently on the probabilistic distribution of the noise. The interpretation of their meaning allows their application to other sciences and engineering, although the results are proven by pure mathematics.   In fact, a mathematical model of  the dynamics of certain physical phenomenon or human-made technology, may be not purely deterministic.   \lq\lq The real world\rq\rq \  which is modeled, for instance by a differential equation, usually behaves (or is perceived  by the observer or is constructed by the engineer), as noisy, exhibiting random perturbations,  more or less near a supposed deterministic model. This noise may be due to multiple causes. For instance,   the physical phenomenon may  need  much more complexity to be completely described than the  variables and parameters that are considered in the simplified  mathematical equations. Instead of taking    its full complexity as it is, it may be convenient to add  certain type of random perturbations to a simplified mathematical model.  Also  the unavoidable inexactitude of the experimental data  and of the observations, from which the mathematical deterministic model was designed, may require the consideration of a noisy dynamics.  Finally, the noisy dynamical systems, and the application of the ergodic theorems that were proved along this paper,  may   mathematically explain better  some physical phenomena, just because epistemologically, the intrinsic nature of them may be not   deterministic, but have predominant random components.
 We cite   Kifer in \cite[Introduction, p. 1]{Kifer}:

  \vspace{.3cm}

  \em \lq\lq  Mathematicians often face the question to which extent mathematical models describe processes of the real world. These models are derived from experimental data, hence they describe real phenomena only approximately. ... Global stability in the presence of noise ... can be described as recovering parameters of dynamical systems from the study of their random perturbations. ... In this way  \em (they) \em can be considered ... having physical sense.\rq\rq \em

\vspace{.3cm}

Finally, we propose some related subjects for futher research:

\noindent{\bf a) } Estimates for the rates of convergence of the time averages for the pointwise convergence ergodic theorems: the abstract tools used in the proof of the ergodic theorems for the transfer operator of noisy systems, may be also used  to extend to stochastic dynamical systems the estimates of the rate of convergence and   deviation, already obtained for deterministic systems  for instance in \cite{Kachurovskii} and \cite{Rey-Bellet-Young}.

\noindent{\bf b) } Relationships between the ergodic components of stationary or periodic measures that satisfy Theorems \ref{MainTheorem1} and \ref{TheoremMainErgodic}, and the spectral properties of the invariant measures supported on the attractor: we propose   the research of the possible extensions to stochastic dynamical systems of the mixing properties, and of the almost periodicity or asymptotic periodicity of the transfer operators, as for instance is proved in \cite{Mezic} and \cite{sandro}  in some particular cases.

\subsubsection{Acknowledgments.}  

The author thanks CSIC of Universidad de la Rep\'{u}blica (Uruguay) for the partial financial support of Project of the group "Sistemas Din\'{a}micos".



\end{document}